\documentclass{amsart}

\usepackage{amssymb}

\usepackage{amsfonts}
\usepackage{mathrsfs}
\usepackage{bbm}
\usepackage{bm}

\theoremstyle{plain}
    \newtheorem{theorem}{Theorem}[section]
    \newtheorem{lemma}[theorem]{Lemma}
    \newtheorem{proposition}[theorem]{Proposition}
    \newtheorem{corollary}[theorem]{Corollary}
\theoremstyle{definition}
    \newtheorem{definition}{Definition}[section]
    \newtheorem{claim}{Claim}
    \newtheorem{remark}{Remark}[section]
    \newtheorem*{acknowledgement}{Acknowledgement}
\theoremstyle{remark}
    \newtheorem*{notation}{Notation}

\numberwithin{equation}{section}

\renewcommand{\l}{\left}
\renewcommand{\r}{\right}
\newcommand{\cleq}{\lesssim}

\def\norm#1{\left\Vert #1 \right\Vert} 

\newcommand{\C}{\mathbb{C}}

\newcommand{\N}{\mathbb{N}}

\newcommand{\R}{\mathbb{R}}

\DeclareMathOperator{\im}{Im}
\DeclareMathOperator{\re}{Re}

\begin{document}

\title[Scattering for NLS system without mass-resonance]{Scattering for the quadratic nonlinear Schr\"{o}dinger system in $\R^5$ without mass-resonance condition}

\author[M.Hamano]{Masaru Hamano}
\address{Course of Mathematics, Programs in Mathematics, Electronics and Informatics, Graduate School of Science and Engineering, Saitama University, Shimo-Okubo 255, Sakura-ku, Saitama-shi, 338-8570, Japan}
\email{m.hamano.733@ms.saitama-u.ac.jp}

\author[T.Inui]{Takahisa inui}
\address{Department of Mathematics, Graduate School of Science, Osaka University, Toyonaka, Osaka, 560-0043, Japan}
\email{inui@math.sci.osaka-u.ac.jp}

\author[K.Nishimura]{Kuranosuke Nishimura}
\address{Department of Mathematics, Graduate School of Science, Tokyo University of Science, 1-3 Kagurazaka, Shinjuku-ku, Tokyo 162-8601, Japan}
\email{1117614@ed.tus.ac.jp}

\maketitle

\begin{abstract}
We consider the quadratic nonlinear Schr\"{o}dinger system (NLS system)
\begin{align*}
	\begin{cases}
		i\partial_t u + \Delta u = v \overline{u}, 
		\\
		i\partial_t v+\kappa \Delta v = u^2, 
	\end{cases}
	\text{ on } I \times \R^5,
\end{align*}
where $\kappa>0$. The scattering below the standing wave solutions for NLS system was obtained by the first author when $\kappa = 1/2$.  The condition of $\kappa=1/2$ is called mass-resonance. In this paper, we prove scattering below the standing wave solutions when $\kappa \neq 1/2$ under the radially symmetric assumption. Our proof is based on the concentration compactness and the rigidity  by Kenig--Merle \cite{KeMe06}. 
Moreover, we discuss the concentration compactness and the rigidity for non-radial solutions. 
\end{abstract}

\tableofcontents


\section{Introduction}
\subsection{Background}

We consider
\begin{align}
\label{NLS'}
	\begin{cases}
	i\partial_t u+\frac{1}{2m}\Delta u=\lambda \overline{u}v,
	\\
	i\partial_t v+\frac{1}{2M}\Delta v=\mu u^2,
	\end{cases}
	\text{ on } I \times \R^d,
\end{align}
where $1\leq d \leq 6$, $(u,v)$ is a $\C^2$-valued unknown function, and $m,M>0$, $\lambda ,\mu \in \C \setminus \{0\}$ are constants.
From physical viewpoint, \eqref{NLS'} is related to the Raman amplification in a plasma. See \cite{CoCoOh09} for details. 
The equation \eqref{NLS'} is invariant under the scaling $\alpha^2(u,v)(\alpha^2t,\alpha x)$ for $\alpha >0$. From this point of view, the critical regularity of the Sobolev space is $s_c=d/2-2$. Therefore, the equation \eqref{NLS'} is called $L^2$-subcritical if $d\leq3$, $L^2$-critical if $d=4$, $\dot{H}^{1/2}$-critical if $d=5$, and $\dot{H}^1$-critical if $d=6$. If $\lambda =c\bar{\mu}$ for some $c>0$, then the mass and the energy are conserved.
In this paper, we focus on the $\dot{H}^{1/2}$-critical case with conservation laws, i.e., $d=5$ and $\lambda =c\bar{\mu}$.
By considering the equation for $(\sqrt{c}|\mu |u(t,x/\sqrt{2m}),c\bar{\mu}v(t,x/\sqrt{2m}))$, we may assume $m=1/2$, $\lambda =\mu =1$.
Thus, we consider the following quadratic nonlinear Schr\"{o}dinger system: 
\begin{align}
\label{NLS}
\tag{NLS}
	\begin{cases}
		i\partial_t u + \Delta u = v \overline{u}, \qquad \text{ on } I \times \R^5,
		\\
		i\partial_t v+\kappa \Delta v = u^2,  \qquad \text{ on } I \times \R^5,
		\\
		(u(0),v(0))=(u_0,v_0) \in H^1(\R^d)\times H^1(\R^5),
	\end{cases}
\end{align}
where $\kappa>0$. The equation \eqref{NLS} has three conserved quantities, i.e., the mass, the energy, and the momentum, which are defined by 
\begin{align}
	\tag{Mass} 
	M(u,v)&:= \norm{u}_{L^2}^2 + \norm{v}_{L^2}^2,
	\\
	\tag{Energy}
	E(u,v)&:= \norm{\nabla u}_{L^2}^2 + \frac{\kappa}{2} \norm{\nabla v}_{L^2}^2 + \re \int_{\R^5} \overline{v}u^2 dx,
	\\
	\tag{Momentum}
	P(u,v)&:= \im \int_{\R^5} \left(\bar{u}\nabla u +\frac{1}{2} \bar{v}\nabla v\right) dx.
\end{align}
The local well-posedness in $H^1(\R^5) \times H^1(\R^5)$, and the existence of the ground sate standing wave solutions were shown by Hayashi, Ozawa, and Tanaka \cite{HaOzTa13}. We recall the ground state. The system \eqref{NLS} has a standing wave solution of the form
\begin{align}
\label{eq1.2}
	(u,v)=( e^{it}\phi (x), e^{2it}\psi (x))
\end{align}
with $\R$-valued functions $\phi ,\psi$.
In fact, if \eqref{eq1.2} is a solution of \eqref{NLS}, then $(\phi ,\psi )$ should satisfy the following system of elliptic equations 
\begin{equation}
\label{SE}
\begin{cases}
	-\phi +\Delta \phi =\phi\psi ,& \text{ in  } \R^5,
	\\
	-2\psi +\kappa \Delta \psi =\phi ^2, & \text{ in  } \R^5.
\end{cases}
\end{equation}
The solution of this elliptic equation \eqref{SE}  can be characterized by the variational argument. Namely, the minimal mass-energy solutions exist and they are called ground states. Roughly speaking, they are characterized by the Pohozaev functional $K$, which is defined by
\begin{align*}
	K( u, v )
	:= \norm{\nabla u}_{L^2}^2+\frac{\kappa}{2} \norm{\nabla v}_{L^2}^2 
	+\frac{5}{4}\re \int_{\R^5} \overline{v}u^2dx.
\end{align*}
We note that $K(\phi,\psi)=0$ if $(\phi ,\psi )$ is a solution of \eqref{SE}. 
Recently, under the mass-resonance condition $\kappa =1/2$, scattering below the ground state was obtained by the first author \cite{Ham18p}, where scattering means that the solution of nonlinear system \eqref{NLS} approaches to a free solution to the Schr\"{o}dinger equations as time goes to infinity. The first author also proved the blow-up or grow-up result below the ground state in \cite{Ham18p}. 
In the mass-resonance case, \eqref{NLS} has the Galilean invariance. That is, $(e^{ix\cdot \xi}e^{-it|\xi|^2} u(t,x-2t\xi), e^{2ix\cdot \xi} e^{-2it|\xi|^2}v(t,x-2t\xi))$ for any $\xi \in \mathbb{R}^5$ is solution to \eqref{NLS} if $(u,v)$ is a solution. This invariance plays an important role in the argument of the first author \cite{Ham18p}.
Roughly speaking, \eqref{NLS} with the mass-resonance condition is similar to the single nonlinear Schr\"{o}dinger equation, whose global dynamics were investigated by many researchers (see e.g. \cite{HoRo08,FaXiCa11,AkNa13,DoMu17,DoMu17p} and references therein).  

On the other hand, it is not clear whether \eqref{NLS} is similar to the single one in general. Thus, there are few works related to the global dynamics of the NLS system in the non mass-resonance case. Recently, the second author, Kishimoto, and the third author \cite{InKiNi18sp} obtained the scattering below the ground state in the $L^2$-critical case, i.e., $d=4$, under the assumption of radial symmetry. Moreover, they also showed the finite time blow-up result below the ground state under the assumption of radial symmetry in \cite{InKiNi18bp}. 

In the present paper, we are interested in the scattering below the ground state for \eqref{NLS} in the non mass-resonance case. 

\subsection{Main results}

In this section, we give main results in this paper. 
We obtain the following scattering result for the radial solutions.

\begin{theorem}
\label{thm1.1}
Let $\kappa \neq 1/2$, and $(\phi,\psi)$ denote a ground state. Assume that $(u_0,v_0) \in H^1(\R^5) \times H^1(\R^5)$ is radially symmetric and satisfies $E(u_0,v_0)M(u_0,v_0)<E(\phi,\psi)M(\phi,\psi)$ and $K(u_0,v_0)\geq 0$. Then, the solution scatters. That is, the solution exists globally in time and there exist $(u_{\pm},v_{\pm}) \in H^1(\R^5) \times H^1(\R^5)$ such that 
\begin{align*}
	\| (u(t),v(t)) -(e^{it\Delta} u_{\pm}, e^{it\kappa \Delta} v_{\pm}) \|_{H^1 \times H^1} \to 0 \text{ as } t \to \pm \infty.
\end{align*} 
\end{theorem}

\begin{remark}
In the opposite case , i.e., $K(u_0,v_0)< 0$, the second author, Kishimoto, and the third author \cite{InKiNi18bp} show that the finite time blow-up occurs in both time directions. And so, the behavior of the radially symmetric solution to \eqref{NLS} below the ground state completely determined by the sign of the functional $K$ at initial time. 
\end{remark}

We apply the argument by Kenig--Merle \cite{KeMe06} to show Theorem \ref{thm1.1}. Their argument is as follows. We suppose that the theorem fails. Then, we can construct a non-scattering global solution with compact orbit, which is called critical element, below the ground state by the concentration compactness argument. After that, by using the virial identity, we can show such solution must be $0$. Such statement is called rigidity theorem. This shows a contradiction. 

In fact, Theorem \ref{thm1.1} is obtained as a corollary of the following propositions. We note that it is enough to consider the positive time directions in the following propositions by the time reversibility of \eqref{NLS}. 

\begin{proposition}[Existence of a critical element]
\label{prop1.2}
Let $\kappa \neq 1/2$. If the statement 
``If $E(u_0,v_0)M(u_0,v_0)<E(\phi,\psi)M(\phi,\psi)$ and $K(u_0,v_0)\geq 0$, then the solution scatters"
is not true, then there exists non-scattering global solution $(u^c,v^c)$ such that $E(u^c_0,v^c_0)M(u^c_0,v^c_0)<E(\phi,\psi)M(\phi,\psi)$, $K(u^c_0,v^c_0)> 0$, the orbit $\{(u^c(t, \cdot+x(t)), v^c(t, \cdot+x(t))): t \in [0,\infty)\}$ is precompact in $H^1(\R^5)\times H^1(\R^5)$ for some continuous function $x: [0,\infty) \to \R^5$. 
\end{proposition}

\begin{proposition}[Rigidity]
\label{prop1.3}
If $(u,v)$ is a non-scattering global solution such that $E(u_0,v_0)M(u_0,v_0)<E(\phi,\psi)M(\phi,\psi)$, $K(u_0,v_0)\geq 0$, the orbit $\{(u(t, \cdot+x(t)), v(t, \cdot+x(t))): t \in  [0,\infty)\}$ is precompact in $H^1(\R^5)\times H^1(\R^5)$ for some continuous function $x: [0,\infty) \to \R^5$, and $P(u_0,v_0) = 0$, then $(u,v) \equiv (0,0)$. 
\end{proposition}

\begin{remark}
Propositions \ref{prop1.2} and \ref{prop1.3} are equivalent to respectively equivalent to Proposition \ref{EC_0} and Proposition \ref{rigidity} below by variational argument. See Section \ref{sec2.2} for the detail. 
\end{remark}

We emphasize that we do not assume radial symmetry in Propositions \ref{prop1.2} and \ref{prop1.3}. If the solution is radially symmetric, we can take $x(t)\equiv 0$ and $P=0$ in these propositions. Thus, Theorem \ref{thm1.1} immediately follows from Propositions \ref{prop1.2} and \ref{prop1.3}

\begin{remark}
The above rigidity theorem is incomplete in the sense that we exclude the case of $P \neq 0$. That is why the scattering for non radially symmetric solutions is an open problem. 
\end{remark}


\subsection{Difficulty of Proof} 

As stated before, we apply the concentration compactness and rigidity by Kenig--Merle \cite{KeMe06}. 
We have a difficulty to show rigidity. 
The difficulty is that \eqref{NLS} is not invariant by Galilean transformation when $\kappa \neq1/2$. 
By the concentration compactness argument, we get a critical element, whose orbit $\{(u(t, \cdot+x(t)), v(t, \cdot+x(t))): t \in [0,\infty)\}$ is precompact for some $x:[0,\infty) \to \R^5$. For the single equation or \eqref{NLS} with $\kappa=1/2$, we can control $x(t)$, which  roughly denotes the center of the critical element, since it is shown from Galilean invariance that the momentum of the critical element is zero. Concretely, we can show that $x(t)/t = o(1)$ as $t \to \infty$. This property is important when we derive a rigidity theorem. In the case of $\kappa \neq 1/2$, we could not control $x(t)$ as above. 
It holds that
\begin{align*}
	\frac{x(t) -X(t)}{t}\to 0 \text{ as }t \to \infty, \text{ where } X(t)=2\frac{\int_0^t\text{Im}\int_{\mathbb{R}^5}\overline{u}\nabla u+\kappa\overline{v}\nabla vdxds}{M(u,v)}.
\end{align*}
Note that $X(t)=2Pt/M$ when $\kappa =1/2$. In general, the behavior of $X(t)$ is not clear. 
In any case, the center of the critical element moves along $X(t)$. That is why we use a boosted virial identity instead of the usual virial identity. Namely, we use
\begin{align*}
	\frac{d}{dt}  \im \int (x-X(t)) \cdot \l( \bar{u}\nabla u + \frac{1}{2} \bar{v}\nabla v \r) dx
	=2K(u(t),v(t)) -X'(t) P(u,v).
\end{align*}
The sign of the right hand side is not clear since we do not know the behavior of $X(t)$. 
To avoid this difficulty, we erase the last term by assuming zero momentum. 

\section{Proof}

\subsection{Preparation of the proof}

Before starting the proof, we prepare some notations for a convenience and some basic results. 
\begin{notation} Let $U_{\kappa}(t):= (e^{it\Delta}, e^{i\kappa t\Delta})$ denote the free propagation. We denote the kinetic energy and the nonlinear energy by
\begin{align*}
	L(u,v)&:= \|\nabla u\|_2^2 + \frac{\kappa}{2} \|\nabla v\|_2^2, 
	\\
	N(u,v)&:=\re \int u^2 \bar{v} dx.
\end{align*}
Thus $E=L+N$ and $K=L+\frac{5}{4}N$ hold. For $\omega>0$, we set 
\begin{align*}
	I_{\omega}(u,v)&:=\frac{1}{2}E(u,v)+\frac{\omega}{2} M(u,v)
	\\
	J_{\omega}(u,v)&:= \frac{1}{10}L(u,v) + \frac{\omega}{2} M(u,v) = I_{\omega}(u,v) - \frac{2}{5}K(u,v)
\end{align*}
We define $\mathcal{N}(u,v):= (v\bar{u}, u^2)$. 
For any Banach space $X$, we set $ \|(u,v)\|_{X}:= \|u\|_X +\|v\|_X$. We set $S(u,v):=\norm{(u,v)}_{L_t^6([0,\infty):L_x^3)}$.
\end{notation}

In this section, we introduce some basic results. 

\begin{proposition}[standard Strichartz estimates]
\label{str0}
Let $(q,r)$ and $(\tilde{q},\tilde{r})$ satisfy $2\leq q,r,\tilde{q},\tilde{r} \leq \infty$ and $2/q+5/r=5/2=2/\tilde{q}+5/\tilde{r}$. Then it follows that
\begin{align*}
	&\norm{e^{it\Delta} \varphi}_{L^q(\R:L^r)} \cleq \norm{\varphi}_{L^2},
	\\
	&\left\|\int_{t_0}^t e^{i(t-s)\Delta}f(s)\, ds \right\|_{L_t^q (I , L^r_x)} \lesssim \|f\|_{L_t^{\tilde{q}'} (I, L_x^{\tilde{r}'})},
\end{align*}
where $I$ is an interval and $t_0\in I$.
\end{proposition}

\begin{proposition}[Non admissible Strichartz]\label{str1}
Let $I$ be an interval and $t_0\in I$. Then it follows that
\begin{align*}
\left\|\int_{t_0}^t e^{i(t-s)\Delta}f(s)\, ds \right\|_{L_t^6 (I , L^3_x)} \lesssim \|f\|_{L_t^{3} (I, L_x^{\frac{3}{2}})}.
\end{align*}
\end{proposition}

See \cite{Caz03} for Propositions \ref{str0} and \ref{str1}.

\begin{proposition}[Stability] Given any $A\geq 0$, there exist $\varepsilon (A)>0$ and $C(A)>0$ with the following property. If $(u,v) \colon [0,\infty) \times \R^5 \rightarrow \mathbb{C}^2$ is a solution to 
\eqref{NLS}, if $(\tilde{u}, \tilde{v}) \in C([0,\infty), H^1 \times H^1)$ and $(e_1 , e_2) \in C([0,\infty) , H^{-1} \times H^{-1})$ satisfy
\begin{align*}
\begin{cases}
i\partial_t \tilde{u} + \Delta \tilde{u} = \tilde{v}\bar{\tilde{u}} + e_1 ,\\
i\partial_t \tilde{v} + \kappa \Delta \tilde{v} = \tilde{u}^2 +e_2,
\end{cases}
\end{align*}
for a.e. $t>0$, and if 
\begin{align*}
&\|(\tilde{u},\tilde{v})\|_{L_t^6 ([0,\infty), L^3)}\leq A,\quad \|(e_1 ,e_2)\|_{L_t^{3}([0,\infty),L_{x}^{\frac{3}{2}})} \leq \varepsilon \leq \varepsilon (A),\\
&\|U_{\kappa}(t)((u-\tilde{u}, v-\tilde{v})(0))\|_{L_t^6 ([0,\infty), L^3)} \leq \varepsilon \leq \varepsilon (A),
\end{align*}
then $(u,v)\in L_t^6 ([0,\infty), L^3)$ and $\|(u,v)-(\tilde{u}, \tilde{v})\|_{L_t^6 ([0,\infty), L^3)}\leq C \varepsilon.$
\end{proposition}
\begin{proof}This follows by a same argument in  \cite{FaXiCa11}. We omit a detail.
\end{proof}

\begin{lemma}
Let $(u,v)$ be a global solution to \eqref{NLS}. If $\|(u,v)\|_{L_t^6 (\R, L_x^3)} <\infty$, then $(u,v)$ scatters to a free solution, i.e., there exist 
$\Psi^{\pm} \in H^1 \times H^1$ such that
\begin{align*}
\lim_{t \rightarrow \pm \infty} \|(u,v)(t) - U_{\kappa}(t)\Psi^{\pm}\|_{H^1}=0. 
\end{align*}
\end{lemma}
\begin{proof}
This follows from the Strichartz estimate. We omit the detail. See e.g. \cite{Ham18p}.
\end{proof}
\begin{lemma}[Small data scattering]\label{SDS}
Let $(u_0,v_0)$ satisfy $\|U_{\kappa} (t) (u_0,v_0)\|_{L_t^6 ([0,\infty) , L_x^3)}$ $\leq\varepsilon_0$, 
where $\varepsilon_0>0$ is a sufficiently small constant. 
Then the maximal-lifespan solution $(u,v)\colon [0,T^{+}) \rightarrow \mathbb{C}^2$ satisfies $\|(u,v)\|_{L_t^6 ([0,T^{+}),L_x^3 )} \lesssim \|U_{\kappa} (t) (u_0,v_0)\|_{L_t^6 ([0,\infty) , L_x^3)}$.

\begin{proof}
Set $\varepsilon :=\|U_{\kappa} (t) (u_0,v_0)\|_{L_t^6 ([0,\infty) , L_x^3)}$ and $f(t):= \|(u,v)\|_{L_t^6 ([0,t),L_x^3 )}$ for any $t\in [0,T^{+})$. Then, by Proposition \ref{str1} it follows that
\begin{align}\label{bootstrap1}
f(t) \leq C\varepsilon + Cf(t)^2.
\end{align}
Set $A:= \{t\in [0,T^{+}) \mid f(t) \leq 2C \varepsilon\}$. Then, by the continuity of $f$, we have $A$ is a closed subset of $[0,T^{+})$. Furthermore we take $\varepsilon_0 >0$ such that
\begin{align*}
\delta \in (0,\varepsilon_0) \Rightarrow C\delta + C(2C\delta)^2 < 2C\delta.
\end{align*}
Combining this with \eqref{bootstrap1}, we can easily show that $A$ is an open subset of $[0, T^{+})$, and so we have $A=[0,T^{+})$.
\end{proof}
\end{lemma}
\begin{remark}
By the Strichartz estimates, we get
\begin{align*}
\|(u,v)\|_{L_t^{\frac{12}{5}}([0,T^+ ), W_x^{1,3})} 
&\lesssim \|(u_0,v_0)\|_{H^1} +\|\int_{0}^t U_{\kappa}(t-s) \mathcal{N}(u,v)\, ds\|_{L_t^{\frac{12}{5}}([0,T^+ ), W_x^{1,3})}\\
&\lesssim \|(u_0,v_0)\|_{H^1} +\| \mathcal{N}(u,v)\|_{L_t^{\frac{12}{7}}([0,T^+ ), W_x^{1,\frac{3}{2}})}\\
&\leq \|(u_0,v_0)\|_{H^1} +\|(u,v)\|_{L_t^6 ([0,T^+), L_x^{3})}  \|(u,v)\|_{L_t^{\frac{12}{5}}([0,T^+ ), W_x^{1,3})}.
\end{align*}
Combining this with Lemma \ref{SDS}, there exists $\delta_{\text{sd}} \in (0,1)$ such that if $(u_0,v_0)\in H^1\times H^1$ and 
$\|(u_0 ,v_0)\|_{H^1} <\delta_{\text{sd}}$,
then the corresponding solution $(u,v)$ to \eqref{NLS} is global for positive time and 
\begin{align}
\|(u,v)\|_{L_{t,x}^{\frac{14}{5}}([0,\infty))} + \|(u,v)\|_{L_t^6 ([0,\infty) , L_x^3)} + \|(u,v)\|_{L_t^{\infty}([0,\infty) , H^1)} \lesssim \|(u_0,v_0)\|_{H^1}. \label{uni bound}
\end{align}
\end{remark}
The following lemma is used in the proof of Proposition \ref{PS}.
\begin{lemma}\label{uniform orthogonality}Let $(u_0 ,v_0) \in H^1 \times H^1$ satisfy $\|(u_0 ,v_0)\|_{H^1} <\delta_{\text{sd}}$ and let $(u,v)$ be the corresponding solution of $\eqref{NLS}$. If $\{t_n^1\}, \{t_n^2\}
\subset \mathbb{R}$ and $\{x_n^1\}, \{x_n^2\} \subset \R^5$ satisfy $|t_n^1 -t_n^2| +|x_n^1 -x_n^2| \rightarrow \infty$ as $n\rightarrow \infty$, then
\begin{align}
\sup_{t\in \R}|(u(t-t_n^1 , \cdot -x_n^1), u(t-t_n^2 , \cdot -x_n^2))_{H^1}| \rightarrow 0 \quad \text{as $n\rightarrow \infty$}, \label{orthogonal1}\\
\sup_{t\in \R}|(v(t-t_n^1 , \cdot -x_n^1), v(t-t_n^2 , \cdot -x_n^2))_{H^1}| \rightarrow 0 \quad \text{as $n\rightarrow \infty$},\label{orthogonal2}
\end{align}
where $(\cdot , \cdot)_{H^1}$ is the scalar product in $H^1.$
\end{lemma}
\begin{proof}
We only prove \eqref{orthogonal1}. \eqref{orthogonal2} can be proved by a same mannaer.
To prove this, we take a sequence $\{t_n \} \subset  \R$ arbitrarily. Setting $f_n:=|(u(t_n -t_n^1 , \cdot -x_n^1), u(t_n-t_n^2 , \cdot -x_n^2))_{H^1}|$, then our aim is to show that
\begin{align*}
f_n \rightarrow 0 \quad \text{as $n\rightarrow \infty$}.
\end{align*}
For this aim, we take any subsequence of $\{f_n\}$ and represent it by the same symbol. Without loss of generality, we may assume
\begin{align*}
f_{n} =|(u(t_n^1 , \cdot -x_n^1), u(t_n^2 , \cdot -x_n^2))_{H^1}|.
\end{align*}
If $|t_n^1| \nrightarrow \infty$ and $|t_n^2| \nrightarrow \infty$, then passing to a subsequence if necessary, we may assume $t_n^1 \rightarrow t_1$ and $t_n^2 \rightarrow t_2$ for 
some $t_1 ,t_2 \in \R$. Then we obtain 
\begin{align*}
f_n  
&\leq \|u(t_n^1)- u(t_1)\|_{H^1} \|u(t_n^2)\|_{H^1} 
\\
&\quad + \|u(t_1)\|_{H^1} \|u(t_n^2) -u(t_2)\|_{H^1} + |(u(t_1,\cdot-x_n^1), u(t_2 , \cdot -x_n^2))_{H^1}|.
\end{align*} 
By \eqref{uni bound} and $|x_n^1 -x_n^2| \rightarrow \infty$, we get $f_n \rightarrow 0$. If $|t_n^1| \rightarrow \infty$ and $|t_n^2| \nrightarrow \infty$, then passing to a subsequence if necessary, we may assume $t_n^1 \rightarrow +\infty$ or $t_n^1 \rightarrow -\infty$ and $t_n^2 \rightarrow t_2$ for some $t_2 \in \R$. We only consider the case $t_n^1 \rightarrow -\infty$. Since $\|(u_0 ,v_0)\|_{H^1} <\delta_{\text{sd}}$, there exists $(\phi ,\psi) \in H^1 \times H^1$ such that
\begin{align*}
\|(u,v)(-t) - U_{\kappa} (-t) (\phi, \psi)\|_{H^1} \rightarrow 0 \quad \text
{as $t\rightarrow \infty$}.
\end{align*}
Then we obtain that
\begin{align*}
f_{n} \leq |(e^{it_n^1\Delta} \phi (\cdot -x_n^1), u(t_2 , \cdot - x_n^2))_{H^1} |+o_{n}(1).
\end{align*}
The result follows from the density argument and the dispersive estimate. The case $|t_n^1|, |t_n^2| \rightarrow \infty$ can be treated by a same manner. We omit the detail.
\end{proof}
\begin{lemma}[Final state problem]\label{FS}
Let $\Psi \in H^1 \times H^1$. Then there exists a solution $(u,v)$ to \eqref{NLS} which is defined on a neighborhood of $-\infty$ and scatters to $U_{\kappa}(-t)\Psi$ as $t\rightarrow +\infty$.
\end{lemma}

\begin{proof}
This follows from a standard argument. We omit the detail. See e.g. \cite{Caz03}. 
\end{proof}

\subsection{Variational argument}
\label{sec2.2}
We define
\begin{align*}
	\mu_{\omega} &:= \inf \{I_{\omega}( f, g ) :( f, g ) \in H^1(\R^5)\times H^1(\R^5) \setminus \{(0,0)\}, K( f, g )=0\}
\end{align*}
for $\omega>0$. Then, it holds that $\mu_{\omega}= I_{\omega}(\phi_{\omega},\psi_{\omega})$, where $(\phi_{\omega}(x),\psi_{\omega}(x))=\omega(\phi(\sqrt{\omega}x), \psi(\sqrt{\omega}x))$ and $( \phi, \psi )$ is a ground state of \eqref{SE}. By the scaling, $(\phi_{\omega}, \psi_{\omega})$ satisfies
\begin{equation*}
\begin{cases}
	-\omega \phi +\Delta \phi =\phi\psi ,& \text{ in  } \R^5,
	\\
	-2\omega \psi +\kappa \Delta \psi =\phi ^2, & \text{ in  } \R^5.
\end{cases}
\end{equation*}
$(e^{i\omega t}\phi_{\omega}, e^{i2\omega t}\psi_{\omega})$ is a solution of \eqref{NLS} for $\omega>0$. 
It is known that $E(u_0,v_0)M(u_0,v_0)<E(\phi,\psi)M(\phi,\psi)$ holds if and only if $I_{\omega}(u_0,v_0) < I_{\omega}(\phi_{\omega},\psi_{\omega})$ for some $\omega>0$. See e.g. \cite{Inu17}.  Thus, Propositions \ref{prop1.2} and \ref{prop1.3} are respectively equivalent to Proposition \ref{EC_0} and Proposition \ref{rigidity} below.

\begin{lemma}\label{a}
The following identity is true. 
\begin{align}
I_{\omega}(\phi_{\omega}, \psi_{\omega}) = \inf \{J_{\omega}(\phi , \psi) \mid (\phi , \psi)\in H^1 \times H^1 \setminus \{(0,0)\}  \text{ and } K(\phi , \psi)\leq 0\}.\label{eq1}
\end{align}
\end{lemma}

\begin{proof}
Set $a:=\text{(RHS) of \eqref{eq1}}$. Then $I_{\omega}(\phi_{\omega}, \psi_{\omega}) \geq a$ is clear. To prove $I_{\omega}(\phi_{\omega}, \psi_{\omega}) \leq a$, we take any $(u,v)\in H^1 \times H^1 $ 
with $K(u,v)< 0$. Setting $f(\lambda):=K(\lambda^{\frac{5}{2}} (u(\lambda \cdot), v(\lambda \cdot) ))$, then we have $f(1)<0$ and 
\begin{align*}
f(\lambda) =\lambda^2 L(u,v) +\frac{5}{4}\lambda^{\frac{5}{2}} N(u,v).
\end{align*}
Since $f(\delta)>0$ for sufficiently small $\delta>0$, there exists $\lambda_0 \in (0,1)$ such that $f(\lambda_0)=0$. Then we get
\begin{align*}
I_{\omega}(\phi_{\omega}, \psi_{\omega}) 
&\leq I_{\omega}(\lambda_0^{\frac{5}{2}} (u(\lambda_0 \cdot), v(\lambda_0 \cdot) ))
\\
&=J_{\omega}(\lambda_0^{\frac{5}{2}} (u(\lambda_0 \cdot), v(\lambda_0 \cdot) ))
\\
&=\frac{\lambda_0^2}{10} L(u,v) + \frac{\omega}{2}M(u,v)
\\
&<J_{\omega}(u,v).
\end{align*}
This implies the result.
\end{proof}

\begin{lemma}[Compatibility of $K$ and $I_{\omega}$]\label{comp}
Let $(u,v)\in H^1 \times H^1$ stisfies $K(u,v)>0$. Then it follows that
\begin{align*}
I_{\omega}(u,v) > \frac{\omega}{2} M(u,v) + \frac{1}{10} L(u,v).
\end{align*}
\begin{proof}
From assumptions, we can easily calculate as follows ( we omit $(u,v)$ ):
\begin{align*}
I_{\omega}(u,v)=\frac{1}{2}E + \frac{\omega}{2} M=\frac{1}{2}(L + N) +\frac{\omega}{2}M>\frac{1}{2}(L -\frac{4}{5}L) +\frac{\omega}{2}M.
\end{align*}
\end{proof}
\end{lemma}
\begin{lemma}\label{est of K}
Let $(u_0 , v_0)\in H^1 \times H^1$ and $(u,v)$ be the maximal-lifespan solution with $(u,v)(0)=(u_0,v_0)$. Assume that
\begin{align*}
I_{\omega}(u_0, v_0) <I_{\omega}(\phi_{\omega}, \psi_{\omega})\quad \text{and} \quad K(u_0,v_0)>0.
\end{align*}
Then it follows that
\begin{align*}
K(u,v)(t) > 0 \quad \text{for any $t$ which is in the lifespan.}
\end{align*}
\end{lemma}
\begin{proof}
If $K(u,v)(t_0)\leq 0$ for some $t_0$. Then there exists $t_1$ such that $K(u,v)(t_1)=0$, and so we have
\begin{align*}
I_{\omega}(u,v)(t_1) \geq I_{\omega}(\phi_{\omega}, \psi_{\omega}).
\end{align*}
This is a contradiction by conservation of mass and energy.
\end{proof}
\begin{remark}
Combining Lemma \ref{comp} and Lemma \ref{est of K}, then we have global existence result under the same assumption in Lemma \ref{est of K}.
\end{remark}

\begin{lemma}
\label{positivityK}
Let $(u_0 , v_0)\in H^1 \times H^1$  satisfy
\begin{align*}
I_{\omega}(u_0, v_0) <I_{\omega}(\phi_{\omega}, \psi_{\omega})\quad \text{and} \quad K(u_0,v_0)>0.
\end{align*}
Then it follows that
\begin{align*}
K(u,v)(t) > \min \l\{ \frac{1}{8} ( I_{\omega}(\phi_{\omega}, \psi_{\omega}) - I_{\omega}(u_0,v_0) ), \delta L(u,v)(t) \r\}
\end{align*}
for some constant $\delta>0$ for all time $t \in \R$.
\end{lemma}

\begin{proof}
The standard calculation obeys this lemma. We omit the details. See \cite[Lemma 3.1]{Ham18p} for details. 
\end{proof}

\begin{corollary}
Let $\Psi$ and $(u,v)$ be as in Lemma \ref{FS}. Suppose also that $\Psi \neq (0,0)$ and 
\begin{align*}
\frac{\omega}{2} M(\Psi) + \frac{1}{2}L(\Psi) < I_{\omega} (\phi_{\omega} , \psi_{\omega}).
\end{align*}
Then $(u,v)$ exists globally and satisfies
\begin{align*}
I_{\omega} (u,v) < I_{\omega} (\phi_{\omega} ,\psi_{\omega}), \quad K(u,v)>0.
\end{align*}
\end{corollary}

\begin{proof}
This is an immediate consequence of Lemma \ref{FS}, conservation of $I_{\omega}$, and the dispersive estimate.
\end{proof}

\subsection{Concentration compactness}
At the beginning of this section, we recall the concentration compactness for the single equation. 

\begin{lemma}[Concentration compactness for single case \cite{FaXiCa11} ]\label{ccs}
Let $a>0$ and let $\{v_n\} \subset H^1$ satisfy
\begin{align*}
\limsup_{n\rightarrow \infty} \|v_n\|_{H^1} \leq a <\infty.
\end{align*}
If 
\begin{align*}
\|e^{it\Delta} v_n\|_{L_{t}^{\infty}((0,\infty) , L^3)} \rightarrow A \quad \text{as $n\rightarrow \infty$},
\end{align*}
then, passing to a subsequence if necessary, there exist $\{(t_n, x_n) \} \subset [0,\infty)\times \mathbb{R}^5$. $\psi \in H^1$, and $\{w_n\} \subset H^1$ such that
\begin{align*}
v_n = e^{-it_n \Delta }\psi (\cdot -x_n) +w_n,
\end{align*}
with 
\begin{align*}
&e^{it_n \Delta} v_n (\cdot +x_n) \rightharpoonup \psi \quad \text{in $H^1$ as $n\rightarrow \infty$},\\
&e^{it_n \Delta} w_n (\cdot + x_n) \rightharpoonup 0 \quad \text{in $H^1$ as $n\rightarrow \infty$},\\
&\|v_n\|_{\dot{H}^{\lambda}}^2 -\|\psi\|_{\dot{H}^{\lambda}}^2 -\|w_n\|_{\dot{H}^{\lambda}}^2 \rightarrow 0 \quad \text{as $n\rightarrow \infty$},
\end{align*}
for all $\lambda \in [0,1]$. Moreover, 
\begin{align*}
\|\psi\|_{H^1} \gtrsim A^{p} a^{-q}, 
\end{align*}
where $p$ and $q$ are positive constant. 
\end{lemma}

Applying Lemma \ref{ccs}, we derive the following concentration compactness for the linear system.

\begin{lemma}[Concentration compactness for the system]
Let $a>0$ and let $\{(u_n,v_n)\}  \subset H^1\times H^1$ satisfy
\begin{align*}
\limsup_{n\rightarrow \infty} \|(u_n,v_n)\|_{H^1} \leq a <\infty.
\end{align*}
If 
\begin{align*}
\|U_{\kappa}(t)(u_n ,v_n)\|_{L_{t}^{\infty}((0,\infty) , L^3)} \rightarrow A \quad \text{as $n\rightarrow \infty$},
\end{align*}
then, passing to a subsequence if necessary, there exist $\{(t_n, x_n) \} \subset [0,\infty)\times \mathbb{R}^5$, $\Psi := (\phi, \psi)\in H^1 \times H^1$, and $\{W_n:= (\zeta_n , w_n)\} \subset H^1 \times H^1$ such that
\begin{align*}
(u_n ,v_n) = U_{\kappa} (-t_n) \Psi (\cdot -x_n) + W_n,
\end{align*}
with 
\begin{align}
\notag
&U_{\kappa}(t_n) (u_n ,v_n)(\cdot +x_n) \rightharpoonup \Psi \quad \text{in $H^1\times H^1$ as $n\rightarrow \infty$},\\
\notag
&U_{\kappa}(t_n) W_n (\cdot + x_n) \rightharpoonup 0 \quad \text{in $H^1 \times H^1$ as $n\rightarrow \infty$},\\
\notag
&\|u_n\|_{\dot{H}^{\lambda}}^2 -\|\phi\|_{\dot{H}^{\lambda}}^2 -\|\zeta_n\|_{\dot{H}^{\lambda}}^2 \rightarrow 0 \quad \text{as $n\rightarrow \infty$},\\
\notag
&\|v_n\|_{\dot{H}^{\lambda}}^2 -\|\psi\|_{\dot{H}^{\lambda}}^2 -\|w_n\|_{\dot{H}^{\lambda}}^2 \rightarrow 0 \quad \text{as $n\rightarrow \infty$},\\
&N(u_n ,v_n) - N(U_{\kappa}(-t_n)\Psi) -N(W_n) \rightarrow 0 \quad \text{as $n\rightarrow \infty$} \label{decoupling of nonlinear term}
\end{align}
for all $\lambda \in [0,1]$. Moreover, 
\begin{align*}
\|\Psi \|_{H^1} \gtrsim A^{p} a^{-q}, 
\end{align*}
where $p$ and $q$ are positive constant. 
\end{lemma}

\begin{proof}
Applying Lemma \ref{ccs}, we easily get the result except for \eqref{decoupling of nonlinear term}. To prove \eqref{decoupling of nonlinear term}, we set 
\begin{align*}
f_n :=| N(u_n ,v_n) - N(U_{\kappa}(-t_n)\Psi) -N(W_n)|.
\end{align*}
Then by the decomposition, we obtain that
\begin{align*}
f_n
&\lesssim  \int (|e^{-it_n \Delta} \phi (\cdot -x_n)| + |\zeta_n |) |e^{-i\kappa t_n \Delta} \psi (\cdot -x_n)||\zeta_n|
\\& \qquad +(|e^{-it_n \Delta} \phi (\cdot -x_n)| + |\zeta_n |) | e^{-it_n \Delta}\phi (\cdot -x_n)| | w_n| dx
\\
&=:f_n^1 + f_n^2.
\end{align*}
Since $f_n^2$ can be treated in the same manner, we only consider $f_n^1$.
Setting $h_n := |e^{-i t_n \Delta} \phi(\cdot -x_n)| + |\zeta_n|$, we obtain
\begin{align*}
\|h_n\|_{L^3} \lesssim \|\phi\|_{H^1} + \|\zeta _n\|_{H^1} \lesssim 1.
\end{align*}
To finish the proof, we take any subsequence $\{f_{n_k}^1\} \subset \{f_n^1\}$. Passing to a subsequence if necessary we may assume $t_{n_k} \rightarrow t_0 \in [0,\infty]$.\\
Case 1. $t_0 =+\infty$. In this case we have by the dispersive estimate 
\begin{align*}
f_{n_k}^1 \lesssim \|e^{-i\kappa t_{n_k} \Delta} \psi (\cdot -x_{n_k})\|_{L^3} \|\zeta_{n_k}\|_{L^3} \|h_{n_k}\|_{L^3}\lesssim \|e^{-
i\kappa t_{n_k} \Delta} \psi \|_{L^3} \rightarrow 0\quad \text{as $k\rightarrow 0$}.
\end{align*}
Case 2. $t_0 \in [0,\infty )$. In this case, we can easily get
\begin{align}\label{weakly to 0}
W_{n_k} (\cdot + x_{n_k}) \rightharpoonup 0 \quad \text{in $H^1\times H^1$}.
\end{align}
Now $\{U_{\kappa} (-t_{n_k})\Psi \}$ is precompact in $H^1 \times H^1$ and so for any $\delta >0$ there exists $R>0$ such that
\begin{align*}
\sup_{k} \|U_{\kappa} (-t_{n_k})\Psi\|_{H^1 (|x|>R)} \leq \delta .
\end{align*}
Therefore we get
\begin{align*}
f_{n_k}^1 
&\lesssim \|e^{-i\kappa t_{n_k} \Delta} \psi\|_{L^3 (|x|>R)} + \|\zeta_{n_k} (\cdot +x_{n_k})\|_{L^3 (|x|<R)}\\
&\leq \delta +\|\zeta _{n_k} (\cdot +x_{n_k})\|_{L^3 (|x|<R)} \rightarrow \delta \quad \text{as $k\rightarrow \infty$},
\end{align*}
where the convergence of second term is followed by \eqref{weakly to 0} and Rellich--Kondrachov theorem. This implies $\limsup_{k\rightarrow \infty} f_{n_k}^1 \lesssim \delta$. Since $\delta >0$ is arbitrarily, we obtain the result.
\end{proof}

\begin{proposition}[Profile decomposition for the system]\label{profile}
Let $\{(u_n ,v_n)\}$ be a bounded sequence in $H^1 \times H^1$. Then, passing to a subsequence if necessary, there exist $J^{\ast} \in \mathbb{Z}_{\geq 0} \cup \{\infty\}$,
 profile $\{\Psi_j := (\phi_j, \psi_j)\}_{j=1}^{J^{\ast}}$, $\{(t_n^j, x_n^j)\} \subset [0,\infty) \times \mathbb{R}^5$, remainder $\{W_n^J := (\zeta_n^J , w_n^J)\}_{J=0}^{J^{\ast}} \subset H^1 \times H^1$ such that for each $J\in \{0,\ldots ,J^{\ast}\}$,
 \begin{align}
 \notag
 &\Psi_j \neq 0 \quad \text{for each $j\in \{1,\ldots, J^{\ast}\}$}, \text{ (there are no profiles if $J^*=0$) }\\
 \notag
 &(u_n ,v_n) = \sum_{j=1}^J U_{\kappa} (-t_n^j) \Psi_j (\cdot -x_n^j) +W_n^J,\\
 &\lim_{J\rightarrow J^{\ast}} \lim_{n\rightarrow \infty} \|U_{\kappa}(t)W_n^J\|_{L_{t}^6([0,\infty), L_x^3)}=0, \label{remainder term}\\
\notag
 &U_{\kappa}(t_n^j)W_{n}^J (\cdot + x_n^j) \rightharpoonup 0\quad \text{weakly in $H^{1} \times H^{1}$ for each $1\leq j\leq J$},\\
 &\lim_{n \rightarrow \infty} [\|u_n\|_{\dot{H}^{\lambda}}^2 - \sum_{j=1}^J \|\phi_j\|_{\dot{H}^{\lambda}}^2 -\|\zeta_n^J\|_{\dot{H}^{\lambda}}^2 ]=0,\label{decoupling1}\\
 &\lim_{n \rightarrow \infty} [\|v_n\|_{\dot{H}^{\lambda}}^2 - \sum_{j=1}^J \|\psi_j\|_{\dot{H}^{\lambda}}^2 -\|w_n^J\|_{\dot{H}^{\lambda}}^2 ]=0,\label{decoupling2}\\
 &\lim_{n\rightarrow \infty} [N(u_n,v_n) - \sum_{j=1}^J N(U_{\kappa}(-t_n^j) \Psi_j) - N(W_n^J)]=0,\label{decoupling3}
 \end{align}
 for all $\lambda \in [0,1]$. 
 Furthermore, for each $j\neq k$ it follows that
 \begin{align}
 |t_n^j -t_n^k|+|x_n^j -x_n^k| \rightarrow \infty \quad \text{as $n\rightarrow \infty$.} \label{asymptotic orthogonality}
 \end{align}
\end{proposition}

\begin{proof}
This follows from standard iteration argument. For details, see Theorem 5.1 in \cite{FaXiCa11}.
\end{proof}

\begin{remark}
We note that we may take $t_n^j$ such that $t_n^j \equiv 0$ or $|t_n^j| \to \infty$ and similarly $x_n^j$ satisfying $x_n^j \equiv 0$ or $|x_n^j| \to \infty$ after some modifications. 
In the radial case, we can show $x_n^j \equiv 0$. We give a rough sketch of the proof. Let $\{(u_n,v_n)\}$ be a bounded sequence of radially symmetric functions in $H^1 \times H^1$. Assume $x_n^j \not\equiv 0$ and thus $|x_n^j| \to \infty$. We take a sequence $\{A_l\} \subset SO(5)$ such that $A_l \neq A_m$ for $l \neq m$. Then, we get $|A_l x_n^j -A_m x_n^j| \to \infty$ as $n \to \infty$. By the profile decomposition, we get $A_l U_{\kappa}(t_n^j) (u_n,v_n)(\cdot+x_n^j)=U_{\kappa}(t_n^j) (u_n,v_n)(\cdot+A_l x_n^j) \rightharpoonup A_l \Psi^j$ weakly in $H^1 \times H^1$. Therefore, we get
\begin{align*}
	\liminf_{n \to \infty} \| (u_n,v_n) \|_{L^2}^2 \geq \sum_{l=1}^{\infty} \| A_l \Psi^j \|_{L^2}^2 =\sum_{l=1}^{\infty} \| \Psi^j \|_{L^2}^2 = \infty.
\end{align*}
This contradicts a boundedness of $\{(u_n,v_n)\}$. 
\end{remark}

We show the decomposition of the functional $I_{\omega}$ and $K$.

\begin{lemma}\label{b}
Let $M \in \mathbb{N}$ and $\{\Psi_j\}_{j=1}^{M} \subset H^1 \times H^1 \setminus \{(0,0)\}$ be such that
\begin{align*}
&\sum_{j=1}^M I_{\omega} (\Psi_j) -\varepsilon \leq I_{\omega}(\sum_{j=1}^M \Psi_j) \leq I_{\omega}(\phi_{\omega},\psi_{\omega}) -\delta\\
&0<K(\sum_{j=1}^M \Psi_j) \leq \sum_{j=1}^M K(\Psi_j) +\varepsilon,  
\end{align*}
where $\delta$ and $\varepsilon$ are positive constant with $2\varepsilon < \delta$. Then we have
\begin{align*}
0<I_{\omega} (\Psi_j) < I_{\omega} (\phi_{\omega} ,\psi_{\omega}), \quad K(\Psi_j) >0
\end{align*}
for any $j\in \{1,\ldots , M\}$.
\end{lemma}
\begin{proof}
We argue by contradiction and so assume $K(\psi_i)\leq 0$ for some $i$. Then it follows from Lemma \ref{a} that
\begin{align*}
I_{\omega} (\phi_{\omega} , \psi_{\omega}) 
&\leq J_{\omega} (\Psi_i)\\
&\leq \sum_{j=1}^M J_{\omega} (\Psi_j)= \sum_{j=1}^M\left( I_{\omega}(\Psi_j) - \frac{2}{5}K(\Psi_j) \right)\\
&\leq I_{\omega}(\sum_{j=1}^M \Psi_j) + \varepsilon +\frac{2}{5}\varepsilon 
<I_{\omega} (\phi_{\omega},\psi_{\omega}).
\end{align*} 
This is a contradiction. Thus we get $K(\Psi_j) >0$ and so $I_{\omega} (\Psi_j) >0$ for any $j\in \{0,\ldots, M\}$ by Lemma \ref{comp}. Then we also get 
\begin{align*}
I_{\omega}(\Psi_j)\leq \sum_{\ell =1}^M I_{\omega} (\Psi_{\ell}) \leq I_{\omega}(\phi_{\omega},\psi_{\omega}) -\delta + \varepsilon <I_{\omega}(\phi_{\omega},\psi_{\omega}).
\end{align*}
\end{proof}

\subsection{Construction of a critical element}

We define critical energy-mass $I_{\omega}^c$ as follows. 
\begin{definition}
\begin{align*}
I_{\omega}^c := \sup \{\delta \in (0, \mu_{\omega}):
&\text{If $(u,v)$ is a maximal-lifespan solution to \eqref{NLS} satisfying} \\
&\text{$K(u,v)>0$ and $I_{\omega}(u,v) <\delta$, then $(u,v)\in L_t^6 ([0,\infty), L^3_x).$} \}
\end{align*}
\end{definition}
Hereafter, we suppose that $I_{\omega}^c < \mu_{\omega} (= I_{\omega}(\phi_{\omega},\psi_{\omega}))$.

Our aim in this section is to prove the following existence result of a critical element.

\begin{proposition}[Existence of a critical element]\label{EC_0}
There exists a maximal-lifespan solution $(u^c,v^c)$ to \eqref{NLS} which satisfies
\begin{align*}
I_{\omega}(u^c,v^c)= I_{\omega}^c< \mu_{\omega},\quad K(u^c_0,v^c_0) >0,
\quad S (u^c ,v^c) =+\infty
\end{align*}
and there exists $x\in C([0,\infty), \mathbb{R}^5)$ such that
 $\{(u^c,v^c)(t, \cdot +x(t)) \mid t\in [0,\infty)\}$ is precompact in $H^1 \times H^1$.
\end{proposition}

To show this, we prepare the following Palais--Smale type result.

\begin{proposition}[Palais--Smale type condition modulo symmetries]\label{PS}
Let $(u_n , v_n)$ be a sequence of maximal-lifespan solutions to \eqref{NLS} such that
\begin{align*}
K (u_n ,v_n) >0 ,\quad I_{\omega}(u_n ,v_n) \rightarrow I_{\omega}^{c},
\end{align*}
and 
\begin{align*}
S (u_n ,v_n) =+\infty \quad \text{for any $n\geq 1.$}
\end{align*}
where $S (u,v):= \|(u,v)\|_{L_t^{6}([0 , \infty ), L_x^3 )}$. Then, passing to a subsequence if necessary, there exist $\Psi \in H^1 \times H^1$, 
$\{W_n\}\subset H^1 \times H^1$, $\{(t_n,x_n)\}\subset [0,\infty)\times \mathbb{R}^5$ 
such that
\begin{align*}
&(u_n ,v_n)(0) = U_{\kappa}(-t_n) \Psi (\cdot -x_n) + W_n, \quad n\geq 1,\\
&t_n \equiv 0 \quad  \text{or} \quad t_n \rightarrow \infty \quad \text{as $n\rightarrow \infty$},\\
&\|W_n\|_{H^1} \rightarrow 0 \quad \text{as $n\rightarrow \infty$}.
\end{align*}
\end{proposition}
\begin{proof}
Passing to a subsequence if necessary, we may assume that
\begin{align*}
I_{\omega}(u_n ,v_n) < I_{\omega}(\phi_{\omega},\psi_{\omega}) -\delta \quad \forall n\geq 1,
\end{align*}
for some $\delta>0.$ Then by Lemma \ref{comp}, $(u_n , v_n)(0)$ is bounded in $H^1 \times H^1$ and so we can apply Proposition \ref{profile}. Let
\begin{align*}
(u_n,v_n)(0) = \sum_{j=1}^J U_{\kappa} (-t_n) \Psi_j (\cdot -x_n^j) +W_n^J
\end{align*}
be the profile decomposition with stated properties. We may assume $t_n^j \equiv 0$ or $\lim_{n\rightarrow \infty} t_n^j=\infty$ for each $j$, after modifying remainder terms.
\footnote{If $J^{\ast}=0$ then we get a contradiction by Lemma \ref{SDS}. So we have at least one profile.} By \eqref{decoupling1}, \eqref{decoupling2}, and \eqref{decoupling3}, passing to a subsequence if necessary and using diagonalization argument, we have
\begin{align*}
&\sum_{j=1}^J I_{\omega} (U_{\kappa}(-t_n^j) \Psi_{j}) + I_{\omega}(W_n^J)-\varepsilon \leq I_{\omega}(u_n (0),v_n (0)) \leq I_{\omega}(\phi_{\omega}, \psi_{\omega})-\delta,\\
&0<K(u_n (0), v_n (0)) \leq \sum_{j=1}^J K(U_{\kappa}(-t_n^j) \Psi_{j}) + K(W_n^J) + \varepsilon,
\end{align*}
for some $\varepsilon \in (0,\delta /2 )$ and any $J, n$. Applying Lemma \ref{b}, we obtain that
\begin{align}
\notag
&0<I_{\omega}(U_{\kappa}(-t_n^j)\Psi_j),\quad I_{\omega}(W_{n}^J) < I_{\omega}(\phi_{\omega}, \psi_{\omega}),\\
&K(U_{\kappa}(-t_n^j)\Psi_j),\quad K(W_n^J) >0, \label{lower bound of K}
\end{align}
for any $n,J,j$. 
\begin{align}
\sum_{j=1}^J I_{\omega} (U_{\kappa}(-t_n^j) \Psi_{j}) 
&=- I_{\omega}(W_n^J) +I_{\omega}(u_n (0), v_n (0)) +o_n (1)\label{decoupling of I}\\
\notag
&\leq I_{\omega}(u_n (0), v_n (0)) +o_n (1).
\end{align}
Thus, by the dispersive estimate, we obtain 
\begin{align}\label{bound1}
\sum_{j=1}^{J}\left( \frac{\omega}{2}M(\Psi_j) + \frac{1}{2}L(\Psi_j) +\frac{1}{2} A_j \right) \leq I_{\omega}^c,
\end{align}
where $A_j$ is defined by 
\begin{align*}
A_j := 
\begin{cases}
N(\Psi_j)\quad &\text{if $t_n^j \equiv 0$},\\
0\quad & \text{if $t_n^j \rightarrow \infty$}.
\end{cases}
\end{align*}
The following lemma is the key step in this proof.
\begin{lemma}\label{only one profile}
$\frac{\omega}{2}M(\Psi_1) + \frac{1}{2}L(\Psi_1) +\frac{1}{2} A_1 = I_{\omega}^c.$
\end{lemma}
\begin{proof}[Proof of Lemma \ref{only one profile}]
If not, there exists $\delta_1 >0$ such that
\begin{align*}
\sup_{1\leq j\leq J^{\ast}} \left( M(\Psi_j) + \frac{1}{2}L(\Psi_j) +\frac{1}{2} A_j \right) <I_{\omega}^c -\delta_1.
\end{align*}
We define nonlinear profiles $(a^j , b^j)$ as follows:
\begin{itemize}
\item If $t_n^{j}\equiv 0$, we define $(a^j ,b^j)$ to be  the maximal-lifespan solution to \eqref{NLS} with $(a^j ,b^j)(0)= \Psi_j$.
\item If $t_n^j \rightarrow \infty$, we define $(a^j , b^j)$ to be the maximal-lifespan solution which scatters  to $U_{\kappa}(-t)\Psi_j$ as $t\rightarrow +\infty$.
\end{itemize}
Using $(a^j ,b^j)$, we define $(a_n^j , b_n^j):=(a^{j}, b^j)(\cdot -t_n^j , \cdot -x_n^j)$. Note that we easily get 
\begin{align}
\notag
&I_{\omega} (a_n^j , b_n^j)(0) = M(\Psi_j) + \frac{1}{2}L(\Psi_j) +\frac{1}{2} A_j <I_{\omega}^c -\delta_1,\\
&\lim_{n\rightarrow \infty}[(u_n , v_n)(0) -\sum_{j=1}^J (a_n^j , b_n^ j )(0) - W_n^J ]=0 \quad \text{in $H^1 \times H^1$}.\label{conv}
\end{align}
Furthermore, by \eqref{bound1} we get
\begin{align*}
\sum_{j=1}^{J^{\ast}}I_{\omega} (a^j , b^j) \leq I_{\omega}^c < \infty,
\end{align*}
and so by small data scattering, there exists $j_{0} \in \mathbb{N}$ such that
\begin{align*}
\|(a_n^j , b_n^j )(0)\|_{H^1} < \delta_{\text{sd}},
\end{align*}
and
\begin{align*}
\|(a_n^j , b_n^j)\|_{L_{t,x}^{\frac{14}{5}} \cap L_{t}^{6}L_{x}^3 \cap L_t^{\infty}H_x^1 (\mathbb{R}^{1+5})}^2
&=\|(a^j , b^j)\|_{L_{t,x}^{\frac{14}{5}} \cap L_{t}^{6}L_{x}^3 \cap L_t^{\infty}H_x^1 (\mathbb{R}^{1+5})}^2
\\
& \lesssim \|(a^j ,b^j)(0)\|_{H^1}^2 \lesssim I_{\omega} (a^j ,b^j),
\end{align*}
for any $j\geq j_0 +1.$ Now we define the approximate solution $(u_n^J , v_n^J)$ as follows:
\begin{align*}
(u_n^J , v_n^J) := \sum_{j=1}^ J (a_n^  j , b_n^j) + U_{\kappa} (t) W_n^J.
\end{align*}
Then $(u_n^J , v_n^J)$ is a solution to following equation
\begin{align*}
\begin{cases}
i\partial_t u_n^J + \Delta u_n^J = v_n^J \bar{u}_n^J + e^J_{1,n},\\
i\partial_t v_n^J + \kappa \Delta v_n^J = (u_n^J)^2+ e^J_{2,n},
\end{cases}
\end{align*}
where $e^J_{1,n} := \sum_{j=1}^J b_n^j \bar{a}_n^j -  v_n^J \bar{u}_n^J$ and $e_{2,n}^J :=\sum_{j=1}^J (a_n^J)^2  -(u_n^J)^2$.
To derive a contradiction, we prepare the following three lemmas.

\begin{lemma}\label{convergence of initial value} For any $J\geq 1$, the following holds:
\begin{align*}
\lim_{n\rightarrow \infty}\|(u_n^J , v_n^J)(0) - (u_n, v_n)(0)\|_{H^1}=0.
\end{align*}
\end{lemma}

\begin{proof}[Proof of Lemma \ref{convergence of initial value}]
This follows immediately from \eqref{conv}.
\end{proof}

\begin{lemma}\label{bound of scattering size}
The following holds:
\begin{align*}
\limsup_{J\rightarrow J^{\ast}} \limsup_{n\rightarrow \infty} S(u_n^J , v_n^J) \leq L_1,
\end{align*}
for some constant $L_1$.
\end{lemma}
\begin{proof}[Proof of Lemma \ref{bound of scattering size}]
Since we know that $\|(a^j ,b^j)\|_{L_t^6 (\R , L_x^3)}  <+\infty$ for each $j$ and $\lim_{J\rightarrow J^{\ast}} \lim_{n\rightarrow \infty} \|U_{\kappa}(t)W_n^J\|_{L_{t}^6([0,\infty), L_x^3)}=0$ hold, it is sufficient to prove that
\begin{align*}
\limsup_{J\rightarrow J^{\ast}} \limsup_{n\rightarrow \infty} S(\sum_{j=j_0 +1}^J (a_n^j , b_n^j) )\lesssim 1.
\end{align*}
Note that by interpolation, we have
\begin{align*}
&\|\sum_{j=j_0 +1}^J (a_n^j , b_n^j)\|_{L_x^3} \lesssim  \|\sum_{j=j_0 +1}^J (a_n^j , b_n^j) \|_{H^1}^{\frac{8}{15}}\|\sum_{j=j_0 +1}^J (a_n^j , b_n^j)\|_{L_x^{\frac{14}{5}}}^{\frac{7}{15}},\\  
&S(\sum_{j=j_0 +1}^J (a_n^j , b_n^j))\lesssim  \|\sum_{j=j_0 +1}^J (a_n^j , b_n^j)\|_{L_t^{\infty} ([0,\infty) , H^1)}^{\frac{8}{15}} \|\sum_{j=j_0 +1}^J (a_n^j , b_n^j)\|_{L_{t}^{\frac{14}{5}}([0,\infty) , L_x^{\frac{14}{5}}) }^{\frac{7}{15}}.
\end{align*}
Therefore it suffices to show that
\begin{align}
\limsup_{J\rightarrow J^{\ast}} \limsup_{n\rightarrow \infty} \|\sum_{j=j_0 +1}^J (a_n^j , b_n^j)\|_{L_t^{\infty} ([0,\infty) , H^1)} \lesssim 1,\label{bound of H1}\\
\limsup_{J\rightarrow J^{\ast}} \limsup_{n\rightarrow \infty} \|\sum_{j=j_0 +1}^J (a_n^j , b_n^j)\|_{L_{t}^{\frac{14}{5}}([0,\infty) , L_x^{\frac{14}{5}}) } \lesssim 1 .\label{bound of Lgamma}
\end{align}
\eqref{bound of H1} can be obtained by Lemma \ref{uniform orthogonality}. Furthermore \eqref{bound of Lgamma} follows by a simple calculation using the asymptotic orthogonality 
condition \eqref{asymptotic orthogonality}.
\end{proof}

\begin{lemma}\label{bound of remainder} The following holds:
\begin{align*}
\lim_{J\rightarrow J^{\ast}} \limsup_{n\rightarrow \infty}\|(e_{1,n}^J , e_{2,n}^J)\|_{L_{t}^3([0,\infty),L_{x}^{\frac{3}{2}})} =0.
\end{align*}
\end{lemma}

\begin{proof}
Here, we only treat the first component. The second component can be treated by a same way. From the H\"{o}lder inequality, we get
\begin{align*}
&\|e_{1,n}^J\|_{L_{t}^3([0,\infty),L_{x}^{\frac{3}{2}})} \\
&\leq \sum_{j\neq k}\|a_n^j b_n^k\|_{L_{t}^3L_{x}^{\frac{3}{2}}} + \|e^{it\Delta} w_n^J\|_{L_{t}^6L_{x}^{3}}
\|\sum_{j=1}^J a_n^j\|_{L_{t}^6L_{x}^{3}}\\
&\quad+ \|e^{i\kappa t\Delta} \zeta_n^J\|_{L_{t}^6L_{x}^{3}} \|\sum_{j=1}^J b_n^j\|_{L_{t}^6L_{x}^{3}} 
  + \|e^{it\Delta} w_n^J\|_{L_{t}^6L_{x}^{3}}  \|e^{i\kappa t\Delta} \zeta_n^J\|_{L_{t}^6L_{x}^{3}}
\\
&\lesssim \sum_{j\neq k}\|a_n^j b_n^k\|_{L_{t}^3L_{x}^{\frac{3}{2}}} + \|e^{it\Delta} w_n^J\|_{L_{t}^6L_{x}^{3}}
+ \|e^{i\kappa t\Delta} \zeta_n^J\|_{L_{t}^6L_{x}^{3}}.
\end{align*}
The last inequality follows by the proof of Lemma \ref{bound of scattering size}. Then by the asymptotic orthogonality condition \eqref{asymptotic orthogonality} and \eqref{remainder term}, we obtain the result.
\end{proof}

Combining these lemmas and applying the stability result, we obtain a contradiction
\footnote{Note that we assumed $S(u_n , v_n)=+\infty$.} and so Lemma \ref{only one profile} is showed.
\end{proof}
Return to the proof of Proposition \ref{PS}. Now by Lemma \ref{only one profile} we get $J^{\ast}=1$ and
\begin{align*}
I_{\omega} (U_{\kappa}(-t_n^1) \Psi_{1}(\cdot -x_n^1)) \rightarrow \frac{\omega}{2}M(\Psi_1) + \frac{1}{2}L(\Psi_1) +\frac{1}{2} A_1 = I_{\omega}^c.
\end{align*}
combining this with \eqref{decoupling of I}, we obtain $I_{\omega}(W_n^1) \rightarrow 0$. By \eqref{lower bound of K} and Lemma \ref{comp}, this implies $\|W_n^1\|_{H^1} \rightarrow 0$.
Therefore we complete the proof.
\end{proof}


We prove the former statements in Proposition \ref{EC_0}.

\begin{proposition}[Existence of a critical element]\label{EC}
There exists a maximal-lifespan solution $(u^c,v^c)$ to \eqref{NLS} which satisfies
\begin{align*}
I_{\omega}(u^c,v^c)= I_{\omega}^c,\quad K(u^c,v^c)(0) >0,
\end{align*}
and
\begin{align*}
S (u^c ,v^c) =+\infty.
\end{align*}
\end{proposition}

\begin{proof}
By the definition of $I_{\omega}^c$, there exists a sequence $(u_n,v_n)$ of maximal-lifespan solutions to \eqref{NLS} which satisfies 
\begin{align}
I_{\omega}(\phi_{\omega},\psi_{\omega})>I_{\omega}(u_n, v_n) \searrow I_{\omega}^c,\quad K(u_n ,v_n)>0,\quad \text{and}\quad S_{\geq 0}(u_n, v_n)=+\infty. \label{assumption of EC}
\end{align}
Applying Proposition \ref{PS}, passing to a subsequence if necessary, there exists $\{(t_n, x_n)\}$, $\Psi$, and $\{W_n\}$ which have stated properties.
Let $(a,b)$ be the nonlinear profile assoiated to $\Psi$ and $t_n$ and suppose 
\begin{align*}
\|(a,b)\|_{L_t^{6}(\R , L_x^3)} <+\infty.
\end{align*}
Then we easily get 
\begin{align*}
\|(u_n , v_n)(0) - (a,b)(-t_n , -x_n)\|_{H^1} \rightarrow 0.
\end{align*}
By the stability result, this implies $S(u_n, v_n) <+\infty$ for sufficiently large $n\in \mathbb{N}$.
\footnote{Note that $S(a,b)(\cdot -t_n , \cdot -x_n)\leq \|(a,b)\|_{L_t^{6}(\R , L_x^3)}<+\infty$.}
This contradicts \eqref{assumption of EC}, and so we obtain that
\begin{align*}
\|(a,b)\|_{L_t^{6}(\R , L_x^3)}=+\infty.
\end{align*}
Then following two possibilities exsist:
\begin{itemize}
\item[I.]\quad  $\|(a,b)\|_{L_t^{6}([0,\infty) , L_x^3)}=+\infty$,
\item[II.]\quad  $\|(\bar{a},\bar{b})(-t,x)\|_{L_t^{6}([0,\infty) , L_x^3)}=+\infty.$
\end{itemize}
Then we easily get the result.
\end{proof}

Next, we show the last part of Proposition \ref{EC_0}. 

\begin{proposition}[Almost periodicity modulo translation]\label{almost periodic}
Let $(u^c ,v^c)$ be the critical element which is constructed in Proposition \ref{EC}. Then there exists $x\in C([0,\infty), \mathbb{R}^5)$ such that
 $\{(u^c,v^c)(t, \cdot +x(t)) \mid t\in [0,\infty)\}$ is precompact in $H^1 \times H^1$.
\end{proposition}

Before prooving Proposition \ref{almost periodic}, we prepare the following lemma.

\begin{lemma}\label{compactness lemma}
Let $(u^c ,v^c)$ be the critical element which is constructed in Proposition \ref{EC} and let $\{n_k\} \subset \mathbb{N}$ satisfies $n_k \rightarrow \infty$ as $k\rightarrow \infty$. Then, 
passing to a subsequence if necessary, there exist $\{x_k\} \subset \R^5$ and $(u,v)\in H^1 \times H^1$ such that $(u^c,v^c)(n_k, \cdot -x_k) \rightarrow (u,v)$ in $H^1 \times H^1.$ 
\end{lemma}

\begin{proof}
Set $(u_k,v_k) :=(u^c ,v^c)(\cdot+n_k)$ and applying Proposition \ref{PS}, then passing to a subsequence if necessary, there exist $\Psi \in H^1 \times H^1$, 
$\{W_n\}\subset H^1 \times H^1$, $\{(t_n,x_n)\}\subset [0,\infty)\times \mathbb{R}^5$ such that
\begin{align*}
&(u_k ,v_k)(0) = U_{\kappa}(-t_k) \Psi (\cdot -x_k) + W_k, \quad k\geq 1,\\
&t_n \equiv 0 \quad  \text{or} \quad t_k \rightarrow \infty \quad \text{as $k\rightarrow \infty$},\\
&\|W_k\|_{H^1} \rightarrow 0 \quad \text{as $k\rightarrow \infty$}.
\end{align*}
If we can prove $t_k \nrightarrow \infty$, the result follows clearly. So we suppose $t_k \rightarrow \infty$ and derive a contradiction. 
Here, by symmetry of free Schr\"{o}dinger equation, we obtain that
\begin{align*}
 (\bar{u}_c , \bar{v}_{c})(n_k) = U_{\kappa}(t_k)\bar{\Psi} (\cdot - x_k) + \bar{W}_k,
\end{align*} 
and so
\begin{align*}
\|U_{\kappa}(t)(\bar{u}_c , \bar{v}_c)(n_k)\|_{L_t^6 ([0,\infty) , L_x^3)} \leq \|U_{\kappa}(t+t_k)\bar{\Psi}\|_{L_t^6 ([0,\infty) , L_x^3)}+ \|U_{\kappa}(t)\bar{W}_k\|_{L_t^6 ([0,\infty) , L_x^3)} \rightarrow 0,
\end{align*}
as $k\rightarrow \infty.$ Combining this with the stability result and symmetry of \eqref{NLS}, we get
\begin{align*}
\lim_{k\rightarrow \infty}\| (\bar{u}_c , \bar{v}_c)(-t +n_k)\|_{L_t^6 ([0,\infty) , L_x^3)}=0.
\end{align*}
By the assumption $n_k \rightarrow \infty$, this implies $\|(u^c ,v^c)\|_{L^6_t (\R , L_x^3)} =0$ and so we get a contradiction.
\end{proof}

\begin{proof}[Proof of Proposition \ref{almost periodic}]
First we prove the following claim.
\begin{claim}
There exist $R>0$  and $\{y_n\}\subset \R^5$ such that
\begin{align*}
\int_{|x|<R} |(u^c , v^c)(n,\cdot -y_n)|^2\, dx \geq \frac{3}{4} M(u^c , v^c).
\end{align*}
\end{claim}
\begin{proof}
If not, there exist $\{R_k \}\subset \R_{>0}$ and $\{n_k\} \subset \mathbb{N}$ such that
\begin{align}
R_k \rightarrow \infty \quad  \text{and} \quad \sup_{y\in \R^5} \int_{|x|<R_k} |(u^c ,v^c)(n_k , \cdot -y)|^2 \, dx \leq \frac{3}{4}M(u^c,v^c).\label{claim1}
\end{align}
If $\{n_k\}$ is bounded, then $\{(u^c, v^c) (n_k)\mid k\in \mathbb{N}\}$ is finite subset of $H^1 \times H^1$. Furthermore the mass $M(u^c ,v^c)$ does not depend on $n_k$.
So we obtain that
\begin{align*}
\int_{|x|<A} |(u^c ,v^c)(n_k) |^2 \, dx >\frac{3}{4}M(u^c,v^c)
\end{align*}
for sufficiently large $A>0$. This is a contradiction. Therefore we may assume $n_k \rightarrow \infty$ as $k\rightarrow \infty.$ Applying Lemma \ref{compactness lemma}, passing to a subsequence if necessary, there exist $\{x_k\} \subset \R$ and $(u,v) \in H^1 \times H^1$ such that
\begin{align*}
(u^c, v^c )(n_k , \cdot -x_k) \rightarrow (u,v)\quad \text{in $H^1 \times H^1$ as $k\rightarrow \infty$}.
\end{align*}
Noting that $M(u^c ,v^c)= M(u,v)$, there exists $\rho >0$ such that
\begin{align*}
\int_{|x| < \rho} |(u,v )|^2 \, dx > \frac{3}{4} M(u^c , v^c).
\end{align*}
Then we get 
\begin{align*}
\int_{|x| < \rho } |(u^c , v^c )(n_k \cdot -x_k)|^2 \, dx > \frac{3}{4} M(u^c , v^c)
\end{align*}
for sufficiently large $k$. 
\end{proof}

Next, we prove that
\begin{claim}\label{claim2}
$\{(u^c ,v^c)(n, \cdot -y_n)\mid n\in \mathbb{N}\}$ is precompact in $H^1 \times H^1$.
\end{claim}

\begin{proof}
To prove this, we take a sequence $\{n_k\} \subset \mathbb{N}$ such that $n_k \rightarrow \infty$ arbitrarily. By Lemma \ref{compactness lemma}, there exist $\{x_k\} \subset \R^5$ and 
$(u,v)\in H^1 \times H^1$ such that
\begin{align}\label{claim2-1}
(u^c, v^c )(n_k , \cdot -x_k) \rightarrow (u,v)\quad \text{in $H^1 \times H^1$ as $k\rightarrow \infty$}.
\end{align}
Noting that $M(u^c ,v^c)= M(u,v)$, there exists $\rho >0$ such that
\begin{align*}
\int_{|x| < \rho} |(u,v )|^2 \, dx > \frac{3}{4} M(u^c , v^c).
\end{align*}
Then we get 
\begin{align*}
\int_{|x| < \rho } |(u^c , v^c )(n_k \cdot -x_k)|^2 \, dx > \frac{3}{4} M(u^c , v^c)
\end{align*}
Then we easily obtain that $|y_{n_k} -x_k| \leq R+\rho$. Combining this with \eqref{claim2-1}, we get the result.
\end{proof}
Finally we give the proof of Proposition \ref{almost periodic}. Since $\{(u^c ,v^c)(n, \cdot -y_n)\mid n\in \mathbb{N}\}$ is precompact, we easily get
$\{(u^c ,v^c)(n+1, \cdot -y_n) \mid n\in \mathbb{N}\}$ is also precompact. In other wards $\{(u^c, v^c )(n, \cdot -y_{n-1}) \mid n\geq 2\}$ is precompact.
Then by the same argument in the proof of Claim \ref{claim2}, we obtain 
\begin{align*}
R:= \sup_{n\geq 2} |y_n -y_{n-1}| <\infty.
\end{align*}
Note that
\begin{align*}
E:= \{(u^c ,v^c) (n+t , \cdot - y_n +y) \mid n\geq 1, t\in [0,1], |y| \leq R\}
\end{align*}
is precompact. Connecting $\{y_n\}$ with a line and denoting it by $-x \in C([0,\infty), \R^5)$, we easily get
\begin{align*}
\{(u^c ,v^c)(t, \cdot + x(t)) \mid t\in [0,\infty)\}\subset E.
\end{align*}
This implies the result.
\end{proof}

To apply the compactness, the following formation is useful. 
\begin{lemma}[e.g. \cite{Ham18p}]
\label{Precompactness of the flow implies uniform localization}
Let $(u,v)$ be the solution to \eqref{NLS} such that
\[
K=\{(u(t,\cdot+x(t)),v(t,\cdot+x(t))):t\in [0,\infty)\}
\]
is precompact in $H^1\!\times\!H^1$. Then for each $\varepsilon>0$, there exists $R>0$ so that
\[
\int_{|x-x(t)|>R}\left(|u(x,t)|^2+|v(x,t)|^2+|\nabla u(x,t)|^2+|\nabla v(x,t)|^2+|v(x,t)u(x,t)^2|\right)dx\leq\varepsilon
\]
for all $0\leq t<\infty$.
\end{lemma}

\subsection{Extinction of the critical element with zero momentum}

In this section, we show the following rigidity proposition.

\begin{proposition}[Rigidity]\label{rigidity}
Let $(u_0,v_0)\in H^1 \times H^1$ and $(u,v)$ be the time-global solution to (NLS) with initial data $(u_0,v_0)$. Suppose
\[
I_\omega(u_0,v_0)<\mu_\omega, \quad K(u_0,v_0)\geq0, \quad P(u_0,v_0)=0,
\]
and there exists a continuous path $x(t)$ such that
\[
K=\{(u(t,\cdot+x(t)),v(t,\cdot+x(t))):t\in [0,\infty)\}
\]
is precompact in $H^1\times H^1$. Then, $(u_0,v_0)=(0,0)$.
\end{proposition}

We investigate the behavior of $x(t)$.

\begin{lemma}\label{order of x(t)}
Let $(u,v)$ be a solution to (NLS) defined on $[0,\infty)$ so that $K=\{(u(\cdot+x(t),t),v(\cdot+x(t),t))\}$ is precompact in $H^1\!\times\!H^1$ for some continuous function $x(\cdot)$. Then
\begin{align*}
\frac{x(t)-X(t)}{t}\longrightarrow0\ \ \text{as}\ \ t\rightarrow\infty,
\end{align*}
where
\begin{align*}
X(t)=2\frac{\int_0^t\im \int_{\mathbb{R}^5}\overline{u}\nabla u+\kappa\overline{v}\nabla vdxds}{M(u,v)}.
\end{align*}
\end{lemma}

\begin{proof}
We set $\widetilde{P}(t)=\widetilde{P}(u,v)(t)=\im \int_{\mathbb{R}^5}\overline{u}\nabla u+\kappa\overline{v}\nabla vdx$, and $M=M(u,v)$ for short. We assume that we do not have Lemma \ref{order of x(t)}. Then, there exist $\delta>0$ and a sequence $t_n\rightarrow\infty$ such that $|x(t_n)-X(t_n)|\geq\delta t_n$. Without loss of generality, we assume that $x(0)=0$. We set
\[
\tau_n=\inf\left\{t\geq0:\left|x(t)-X(t)\right|\geq\left|x(t_n)-X(t_n)\right|\right\}.
\]
Since $0<\tau_n\leq t_n$ and $|x(\tau_n)-X(\tau_n)|=|x(t_n)-X(t_n)|$, it follows that
\[
\tau_n\longrightarrow\infty\ \ \text{as}\ \ n\rightarrow\infty,
\]
\[
\left|x(t)-X(t)\right|<\left|x(t_n)-X(t_n)\right|\,,\ \ \ 0\leq t<\tau_n,
\]
\[
\left|x(\tau_n)-X(\tau_n)\right|\geq\delta\tau_n.
\]
Let $\chi_1 \in C_0^\infty(\mathbb{R}^5)$ be radial with
\begin{equation}
\notag \chi_1(r)=
\begin{cases}
\hspace{-0.4cm}&\displaystyle{\ \ \ \ 1\ \ \ \ \ \ \,(0\leq r\leq1),}\\
\hspace{-0.4cm}&\displaystyle{smooth\ \ (1\leq r\leq 2),}\\
\hspace{-0.4cm}&\displaystyle{\ \ \ \ 0\ \ \ \ \ \ \,(2\leq r),}
\end{cases}
\end{equation}
where $r=|x|$. Also, let $\chi_1$ satisfy $|\chi_1'(r)|\leq2\ \ (r\geq0)$. We define $\chi_R(r)= \chi_1(r/R)$ for $R>0$. 
We define
\[
z_R(t)=\int_{\mathbb{R}^5}\left(x-X(t)\right)\chi_R\left(\left|x-X(t)\right|\right)\left(|u(t,x)|^2+|v(t,x)|^2\right)dx
\]
for $R>0$. By a direct calculation, 
we have
\begin{align*}
z_R'(t)
	&=\frac{2\widetilde{P}(t)}{M}\int_{R\leq|x-X(t)|}\left\{1-\chi_R\left(\left|x-X(t)\right|\right)\right\}\left(|u(t,x)|^2+|v(t,x)|^2\right)dx \notag \\
	&\quad -\int_{R\leq|x-X(t)|\leq2R}\left(x-X(t)\right)\chi_R'\left(\left|x-X(t)\right|\right) \frac{\frac{2\widetilde{P}(t)}{M}\cdot \left(x-X(t)\right)}{\left|x-X(t)\right|}\left(|u(t,x)|^2+|v(t,x)|^2\right)dx \notag \\
	&\quad +2\text{Im}\int_{R\leq|x-X(t)|}\left\{\chi_R\left(\left|x-X(t)\right|\right)-1\right\}\left(\nabla u\overline{u}+\kappa\nabla v\overline{v}\right)dx \notag \\
	&\quad +2\text{Im}\int_{R\leq|x-X(t)|\leq2R}\frac{\left(x-X(t)\right)\chi_R'\left(\left|x-X(t)\right|\right)}{\left|x-X(t)\right|}\left(x-X(t)\right)\cdot\left(\nabla u\overline{u}+\kappa\nabla v\overline{v}\right)dx.
\end{align*}
Since we have $|\widetilde{P}(t)|\lesssim I_\omega(u,v)$,
it holds that
\begin{align*}
|z_R'(t)|
&\leq C(M,I_\omega(u,v)) \int_{R\leq |x-X(t)|}\left(|\nabla u|^2+|u|^2+|\nabla v|^2+|v|^2\right)dx.
\end{align*}
On the other hand, by Lemma \ref{Precompactness of the flow implies uniform localization}, there exists $\rho>0$ such that for any $0\leq t<\infty$,
\[
\int_{|x-x(t)|>\rho}\left(|\nabla u|^2+|u|^2+|\nabla v|^2+2|v|^2\right)dx\leq\frac{\delta M(u,v)}{10(1+\delta)} \times \min\left\{\frac{5}{C(M,I_\omega(u,v))},1\right\}.
\]
Let $R_n=|x(\tau_n)-X(\tau_n)|+\rho$. Since for given $0\leq t\leq\tau_n$ and $|x-X(t)|>R_n$,
\[
|x-x(t)|=\left|\left(x-X(t)\right)+\left(X(t)-x(t)\right)\right|\geq R_n-\left|x(t)-X(t)\right|\geq R_n-\left|x(\tau_n)-X(\tau_n)\right|=\rho,
\]
we obtain
\begin{align}
|z'_{R_n}(t)|\leq  C(M,I_\omega(u,v))\int_{|x-x(t)|>\rho}\left(|\nabla u|^2+|u|^2+|\nabla v|^2+|v|^2\right)dx\leq \frac{\delta M(u,v)}{2(1+\delta)}.\label{63}
\end{align}
for any $n\in\mathbb{N}$ and $0\leq t\leq\tau_n$. Also, since $R_n\geq\rho$ and $x(0)=0$,
\begin{align}
|z_{R_n}(0)|&\leq\int_{|x|<\rho}|x|\chi_{R_n}(|x|)\left(|u_0|^2+|v_0|^2\right)dx+\int_{|x|>\rho}|x|\chi_{R_n}(|x|)\left(|u_0|^2+|v_0|^2\right)dx \notag \\
&=\int_{|x|<\rho}|x|\left(|u_0|^2+|v_0|^2\right)dx+\int_{2R_n>|x-x(0)|>\rho}|x|\chi_{R_n}(|x|)\left(|u_0|^2+|v_0|^2\right)dx \notag \\
&\leq\rho M(u,v)+\frac{\delta M(u,v)}{5(1+\delta)}R_n.\label{64}
\end{align}
Next, we estimate $z_{R_n}(\tau_n)$.
\begin{align*}
z_{R_n}(\tau_n)
	&=\int_{|x-x(\tau_n)|>\rho}\left(x-X(\tau_n)\right)\chi_{R_n}\left(\left|x-X(\tau_n)\right|\right)\left(|u(\tau_n,x)|^2+|v(\tau_n,x)|^2\right)dx\\
	&\hspace{1.0cm}+\int_{|x-x(\tau_n)|<\rho}\left(x-X(\tau_n)\right)\chi_{R_n}\left(\left|x-X(\tau_n)\right|\right)\left(|u(\tau_n,x)|^2+|v(\tau_n,x)|^2\right)dx\\
	&=:I+I\!\!I.
\end{align*}
We have
\[
|I|\leq \frac{\delta M(u,v)}{5(1+\delta)}R_n.
\]
If $|x-x(\tau_n)|<\rho$, then we have $|x-X(\tau_n)|\leq|x-x(\tau_n)|+|x(\tau_n)-X(\tau_n)|\leq\rho+|x(\tau_n)-X(\tau_n)|=R_n$. Thus, $\chi_{R_n}(|x-X(\tau_n)|)=1$. Therefore, 
\begin{align*}
-I\!\!I
&=-\int_{|x-x(\tau_n)|<\rho}\left(x-X(\tau_n)\right)\left(|u(\tau_n,x)|^2+|v(\tau_n,x)|^2\right)dx\\
&=-\int_{|x-x(\tau_n)|<\rho}(x-x(\tau_n))\left(|u(\tau_n,x)|^2+|v(\tau_n,x)|^2\right)dx\\
&\hspace{3.5cm}+\left(X(\tau_n)-x(\tau_n)\right)\int_{|x-x(\tau_n)|<\rho}\left(|u(\tau_n,x)|^2+|v(\tau_n,x)|^2\right)dx\\
&=\left(X(\tau_n)-x(\tau_n)\right)M(u,v)-\int_{|x-x(\tau_n)|<\rho}(x-x(\tau_n))\left(|u(\tau_n,x)|^2+|v(\tau_n,x)|^2\right)dx\\
&\hspace{3.5cm}-\left(X(\tau_n)-x(\tau_n)\right)\int_{|x-x(\tau_n)|>\rho}\left(|u(\tau_n,x)|^2+|v(\tau_n,x)|^2\right)dx.
\end{align*}
Hence,
\[
|I\!\!I|\geq\left|X(\tau_n)-x(\tau_n)\right|M(u,v)-\rho M(u,v)-\frac{\delta M(u,v)}{10(1+\delta)}R_n.
\]
Therefore,
\begin{align}
|z_{R_n}(\tau_n)|\geq-|I|+|I\!\!I|\geq\left|X(\tau_n)-x(\tau_n)\right|M(u,v)-\rho M(u,v)-\frac{3\delta M(u,v)}{10(1+\delta)}R_n.\label{65}
\end{align}
Combining \eqref{63},\,\eqref{64}, and \eqref{65},
\begin{align*}
\frac{\delta M(u,v)}{2(1+\delta)}\tau_n&=\int_0^{\tau_n}\frac{\delta M(u,v)}{2(1+\delta)}dt\geq\int_0^{\tau_n}|z'_{R_n}(t)|dt\geq\left|\int_0^{\tau_n}z_{R_n}'(t)dt\right|\\
&\geq|z_{R_n}(\tau_n)|-|z_{R_n}(0)|\geq\left|X(\tau_n)-x(\tau_n)\right|M(u,v)-2\rho M(u,v)-\frac{\delta M(u,v)}{2(1+\delta)}R_n.
\end{align*}
Substituting $R_n=|X(\tau_n)-x(\tau_n)|+\rho$,
\[
\frac{\delta M(u,v)}{2(1+\delta)}\tau_n\geq\left|X(\tau_n)-x(\tau_n)\right|M(u,v)-2\rho M(u,v)- \frac{\delta M(u,v)}{2(1+\delta)}\left(\left|X(\tau_n)-x(\tau_n)\right|+\rho\right),
\]
\[
\frac{\delta}{2(1+\delta)}\tau_n\geq \frac{2+\delta}{2(1+\delta)}\left|X(\tau_n)-x(\tau_n)\right|-\frac{4+5\delta}{2(1+\delta)}\rho,
\]
\[
\frac{\delta}{2+\delta}+\frac{4+5\delta}{2+\delta}\frac{\rho}{\tau_n}\geq \frac{|X(\tau_n)-x(\tau_n)|}{\tau_n}.
\]
Since $0<\tau_n\leq t_n$ and $|X(\tau_n)-x(\tau_n)|=|X(t_n)-x(t_n)|$,
\[
\frac{\delta}{2+\delta}+\frac{4+5\delta}{2+\delta}\frac{\rho}{\tau_n}\geq \frac{|X(t_n)-x(t_n)|}{t_n}.
\]
We obtain $\frac{4+5\delta}{2+\delta}\frac{\rho}{\tau_n}\leq\frac{\delta}{2}$ for sufficiently large $n\in\mathbb{N}$ by $\tau_n\rightarrow\infty$\ \ as\ \ $n\rightarrow\infty$. Also, since $\frac{\delta}{2+\delta}<\frac{\delta}{2}$ by $\delta>0$, it follows that
\[
\frac{|X(t_n)-x(t_n)|}{t_n}<\frac{\delta}{2}+\frac{\delta}{2}=\delta.
\]
This is in contradiction to $|X(t_n)-x(t_n)|\geq\delta t_n$. Therefore, Lemma \ref{order of x(t)} holds.
\end{proof}

\begin{lemma}
\label{positivity}
Let $(u,v)$ be a global solution in Proposition \ref{rigidity} (not necessary $P=0$). Then, there exists a constant $A>0$ such that
\[
A M(u,v)\leq L(u,v),
\]
for all time $t \in [0,\infty)$.
\end{lemma}

\begin{proof}
If $(u,v)=(0,0)$, the statement holds. Let $(u,v) \neq (0,0)$.
For simplicity, we set $L(t):=L(u(t),v(t))$ and $M:=M(u,v)$. If the statement fails, for any $n \in \N$, there exists $t_n>0$ such that
\begin{align*}
	\frac{1}{n} M > L(t_n).
\end{align*}
This implies that $L(t_n) \to 0$ as $n \to \infty$. Taking a subsequence, $(u(t_n),v(t_n))$ converges to $0$ in $H^1 \times H^1$ by the compactness of the orbit. This means $(u,v) = (0,0)$. This is a contradiction. 
\end{proof}


We show Proposition \ref{rigidity}. 

\begin{proof}[Proof of Proposition \ref{rigidity}]
In the case $K(u_0,v_0)=0$, we have $(u_0,v_0)=(0,0)$ by the definition of $\mu_\omega^{20,8}(=I_\omega(\phi_\omega,\psi_\omega))$. Let $K(u_0,v_0)>0$. We lead to contradiction. By Lemma \ref{order of x(t)}, for any $\eta>0$, there exists $T_0=T_0(\eta)>0$ such that
\[
\left|x(t)-X(t)\right|\leq\eta t
\]
for any $t\geq T_0$. For $R>0$, let $\chi_R$ be a radial function as in the proof of Lemma \ref{order of x(t)}.
We define 
\begin{align*}
I(t)=2\text{Im}\int_{\mathbb{R}^5}\chi_R\left(\left|x-X(t)\right|\right)\left(x-X(t)\right)\cdot\left( \nabla u \overline{u}+\frac{1}{2}\nabla v \overline{v}\right)dx.
\end{align*}
Then,
\begin{align*}
|I(t)|&\leq2\left|\int_{\mathbb{R}^5}\chi_R\left(\left|x-X(t)\right|\right)\left(x-X(t)\right)\cdot\left( \nabla u \overline{u}+ \frac{1}{2}\nabla v \overline{v}\right)dx\right|\\
&\leq 4R\int_{|x-X(t)|\leq2R} \left(|\nabla u||u|+\frac{1}{2}|\nabla v||v|\right)dx\\
&\leq \widetilde{c}R
\end{align*}
for any $0\leq t<\infty$. We calculate $I'(t)$.
\begin{align*}
I'(t) 
	&= -2\text{Im}\int_{\mathbb{R}^5}\chi_R'\left(\left|x-X(t)\right|\right) \frac{(x-X(t))\cdot \frac{2\widetilde{P}(t)}{M}}{|x-X(t)|}\left( x-X(t)\right)\cdot\left( \nabla u \overline{u}+ \frac{1}{2}\nabla v \overline{v}\right)dx
	\\
	& \quad-\frac{4\widetilde{P}(t)}{M}\cdot\text{Im}\int_{\mathbb{R}^5} \chi_R\left(\left|x-X(t)\right|\right) \left( \nabla u \overline{u}+ \frac{1}{2}\nabla v \overline{v}\right)dx
	\\
	&\quad+ 2\int_{\mathbb{R}^5}\chi_R\left(\left|x-X(t)\right|\right)\left(x-X(t)\right)\cdot\partial_t\text{Im}\left(\nabla u\overline{u}+\frac{1}{2}\nabla v\overline{v}\right)dx
	\\
	&=:S_1+S_2+J
	\\
	&=:S_1 +\widetilde{S}_2 -\frac{4\widetilde{P}(t)\cdot P}{M} + J
\end{align*}
By the assumption of $P=0$, the third term in the last disappears. 
Now, we want to show $I'(t) \geq 4K(u,v) - \text{error}$. From direct calculations, we obtain
\begin{align*}
|S_1|
	&\leq C(M,I_\omega(u,v))\int_{R\leq|x-X(t)|\leq2R}\left( |\nabla u||u|+ |\nabla v||v|\right)dx\\
	&\leq C(M,I_\omega(u,v))\int_{R\leq|x-X(t)|\leq2R}\left(|\nabla u|^2+|u|^2+|\nabla v|^2+|v|^2\right)dx,
\end{align*}
and
\begin{align*}
|\widetilde{S}_2|
	&=\left|\frac{4\widetilde{P}(t)}{M}\cdot\text{Im}\int_{\mathbb{R}^5} \left\{\chi_R\left(\left|x-X(t)\right|\right)-1\right\} \left( \nabla u \overline{u}+ \frac{1}{2}\nabla v \overline{v}\right)dx\right|\\
	&\leq C(M,I_\omega(u,v))\int_{R\leq|x-X(t)|}\left(|\nabla u|^2+|u|^2+|\nabla v|^2+|v|^2\right)dx.
\end{align*}
Moreover, we also have
\begin{align*}
	J
	&=4K(u,v)
	\\
	&+4 \int_{\mathbb{R}^5}\frac{\chi_R'(|x-X(t)|)}{|x-X(t)|}\left\{|(x-X(t))\cdot\nabla u|^2+\frac{\kappa}{2}|(x-X(t))\cdot\nabla v|^2 \right\}dx \notag \\
	& +4\int_{\mathbb{R}^5}\left\{\chi_R\left(|x-X(t)|\right)-1\right\}\left(|\nabla u|^2+\frac{\kappa}{2}|\nabla v|^2\right)dx \notag \\
	& +\sum_{k=1}^5\sum_{j=1}^5\int_{\mathbb{R}^5}\chi_R\left(\left|x-X(t)\right|\right)\left(x-X(t)\right)\cdot\bm{e_k}\partial_{jjk}\left(|u|^2+\frac{\kappa}{2}|v|^2\right)dx \notag \\
	& +\int_{\mathbb{R}^5}\chi_R'(|x-X(t)|)\,|x-X(t)|\,\text{Re}(v\overline{u}^2)dx \notag \\
	& +5\int_{\mathbb{R}^5}\left\{\chi_R(|x-X(t)|)-1\right\}\text{Re}(v\overline{u}^2)dx,
\end{align*}
where $\bm{e_k}$ denotes the $k$-th standard basis vector.
Estimating these terms except for $4K$ and combining the estimates of $S_1$ and $\widetilde{S}_2$
\begin{align*}
I'(t)
&\geq J-|S_1|-|\widetilde{S}_2|\\
&\geq 4K(u,v)-C(M,I_\omega(u,v))\int_{R\leq|x-X(t)|}(|\nabla u|^2+|u|^2+|\nabla v|^2+|v|^2+|vu^2|)dx
\end{align*}
for $R>1$. By Lemma \ref{Precompactness of the flow implies uniform localization}, there exists $R_0>1$ such that
\begin{align*}
C(M,I_\omega(u,v))\int_{R_0\leq|x-x(t)|}\left(|\nabla u|^2+|\nabla v|^2+|u|^2+|v|^2+|vu^2|\right)dx
\qquad\\  <\frac{1}{4}\min\{I_\omega(\phi_\omega,\psi_\omega)-I_\omega(u,v),cM(u,v)\}
\end{align*}
for any $0\leq t<\infty$. If we take $\displaystyle R=R_0+\sup_{t\in[T_0,T_1]}|X(t)-x(t)|>1$, then $\displaystyle|x-x(t)|\geq|x-X(t)|-|X(t)-x(t)|\geq R-\sup_{t\in[T_0,T_1]}|X(t)-x(t)|=R_0$ for $x$ with $|x-X(t)|>R$ and $t\in[T_0,T_1]$, where $T_1>T_0$ is chosen later. Thus,
\begin{align*}
C(M,&I_\omega(u,v))\int_{R\leq|x-X(t)|}\left(|\nabla u|^2+|\nabla v|^2+|u|^2+|v|^2+|vu^2|\right)dx\\
&\leq C(M,I_\omega(u,v))\int_{R_0\leq|x-x(t)|}\left(|\nabla u|^2+|\nabla v|^2+|u|^2+|v|^2+|vu^2|\right)dx\\
&\leq\frac{1}{4}\min\{I_\omega(\phi_\omega,\psi_\omega)-I_\omega(u,v),cM(u,v)\}.
\end{align*}
Therefore, it holds from the above inequalities and the positivity of $K$, Lemmas \ref{positivityK} and \ref{positivity} that
\begin{align*}
I'(t)&\geq\frac{1}{4}\min\{I_\omega(\phi_\omega,\psi_\omega)-I_\omega(u,v),cM(u,v)\}.
\end{align*}
for $R>1$ and $t\in[T_0,T_1]$. Integrating this inequality in $[T_0,T_1]$,
\begin{align*}
\frac{1}{4}\min\{I_\omega(\phi_\omega,\psi_\omega)-I_\omega(u,v),cM(u,v)\}&(T_1-T_0)
\\&\leq I(T_1)-I(T_0)\\
&\leq|I(T_1)|+|I(T_0)|\\
&\leq2\widetilde{c}R\\
&=2\widetilde{c}\left(R_0+\sup_{t\in[T_0,T_1]}|X(t)-x(t)|\right)\\
&\leq2\widetilde{c}\left(R_0+T_1\eta\right).
\end{align*}
This inequality is contradiction if we take $\eta>0$ sufficiently small and $T_1>0$ sufficintly large.  
\end{proof}

\begin{acknowledgement}
The second author was partially supported by Grant-in-Aid for Early-Career Scientists 18K13444. 
\end{acknowledgement}


\end{document}